\documentclass[a4paper,12pt]{article}
\usepackage{amssymb,amsthm,amsmath,latexsym}

\newtheorem{thm}{Theorem}[section]
\newtheorem{pro}[thm]{Proposition}
\newtheorem{lem}[thm]{Lemma}
\newtheorem{cor}[thm]{Corollary}
\newtheorem{rem}[thm]{Remark}

\def\l{\langle}
\def\r{\rangle}

\date{}

\title{\normalsize{\bf ON LOCALLY GRADED GROUPS\\ WITH A WORD WHOSE VALUES ARE ENGEL}\\
[10pt]
\small{Dedicated to Professor Howard Smith on the occasion of his retirement}}

\author{\small{\textsc{Pavel Shumyatsky\begin{footnote}{This work was carried out during the first author's visit to the University of Salerno, which was partially supported by GNSAGA.}\end{footnote}}}\\
\small{Department of Mathematics, University of Brasilia}\\
\small{Brasilia-DF, 70910-900 Brazil}\\
\small{E-mail: pavel@unb.br}\\
[10pt]
\small{\textsc{Antonio Tortora} and \textsc{Maria Tota}}\\
\small{Dipartimento di Matematica, Universit\`a di Salerno}\\
\small{Via Giovanni Paolo II, 132 - 84084 - Fisciano (SA), Italy}\\
\small{E-mail: antortora@unisa.it, mtota@unisa.it}}

\begin{document}
\maketitle

\begin{abstract} Let $m,n$ be positive integers, $v$ a multilinear commutator word and $w=v^m$. We prove that if $G$ is a locally graded group in which all $w$-values are $n$-Engel, then the verbal subgroup $w(G)$ is locally nilpotent.\\

\end{abstract}

\section{Introduction}

If $w$ is a group-word and $G$ is a group, the verbal subgroup $w(G)$ of $G$ is the subgroup generated by all $w$-values in $G$. Most of the words considered in this paper are {\it multilinear commutators}, also known under the name of {\it outer commutator words}. These are words that have a form of a multilinear Lie monomial, i.e., they are constructed by nesting commutators but using always different variables. For example the word
$$[[x_1,x_2],[y_1,y_2,y_3],z]$$
is a multilinear commutator  while the Engel word
$$[x,y,y,y]$$
is not.

An important family of multilinear commutators consists of the lower central words $\gamma_k$,
given by
\[
\gamma_1=x_1,
\qquad
\gamma_k=[\gamma_{k-1},x_k]=[x_1,\ldots,x_k],
\quad
\text{for $k\ge 2$.}
\]
The corresponding verbal subgroups $\gamma_k(G)$ are the terms of the lower central series of $G$.  Another distinguished sequence of outer commutator words are the derived words $\delta_k$, on $2^k$ variables, which are defined by
\[
\delta_0=x_1,
\quad
\delta_k=[\delta_{k-1}(x_1,\ldots,x_{2^{k-1}}),\delta_{k-1}(x_{2^{k-1}+1},\ldots,x_{2^k})],
\quad
\text{for $k\ge 1$.}
\]
The verbal subgroup that corresponds to the word $\delta_k$ is the familiar $k$-th derived subgroup of $G$ usually denoted by $G^{(k)}$.

Let $n$ be a positive integer and let $x,y$ be elements of a group $G$. The commutators $[x,_n y]$ are defined inductively by the rule
$$[x,_0 y]=x,\quad [x,_n y]=[[x,_{n-1} y],y].$$
An element $x$ is called a (left) Engel element if for any $g\in G$ there exists $n=n(x,g)\geq 1$ such that $[g,_n x]=1$. If $n$ can be chosen independently of $g$, then $x$ is a (left) $n$-Engel element. A group $G$ is called $n$-Engel if all elements of $G$ are $n$-Engel. It is a long-standing problem whether any $n$-Engel group is locally nilpotent. Following Zelmanov's solution of the restricted Burnside problem \cite{ze1,ze2}, Wilson proved that this is true if $G$ is residually finite \cite{w}. Later the first author showed that if in a residually finite group $G$ all commutators $[x_1,\ldots,x_k]$ are $n$-Engel, then $\gamma_k(G)$ is locally nilpotent \cite{shu0,shu2}. In the recent paper \cite{BSTT} a stronger result was obtained. Namely, it was proved that given positive integers $m,n$ and a multilinear commutator word $v$, if $G$ is a residually finite group in which all values of the word $w=v^m$ are $n$-Engel, then the verbal subgroup $w(G)$ is locally nilpotent.

In \cite{kr}, Kim and Rhemtulla extended Wilson's theorem by showing that any locally graded $n$-Engel group is locally nilpotent. Recall that a group is locally graded if every non-trivial finitely generated subgroup has a proper subgroup of finite index. The class of locally graded groups is fairly large and in particular it contains all residually finite groups. The purpose of the present paper is to extend the main result of \cite{BSTT} to locally graded groups. Thus, we will prove the following theorem.

\begin{thm}\label{main}
Let $m,n$ be positive integers, $v$ a multilinear commutator word and $w=v^m$. If $G$ is a locally graded group in which all $w$-values are $n$-Engel, then the verbal subgroup $w(G)$ is locally nilpotent.
\end{thm}

Earlier the result was proved in \cite{BSTT} under the additional hypothesis that $m=1$. Our proof of Theorem \ref{main} is based on the techniques that Zelmanov created in his solution of the restricted Burnside problem. In particular a result in the spirit of the restricted Burnside problem is given in Corollary \ref{residual} below. It provides a sufficient condition under which a group must be locally finite.

It remains unclear whether Theorem \ref{main} can be extended to arbitrary words.
\section{The proof}

We say that a set $X$ is  {\it commutator-closed} if $[x,y]\in X$ whenever $x,y\in X$. Furthermore, given subgroups $H$ and $K$ of a group $G$, we denote by $H^K$ the smallest subgroup of $G$ containing $H$ and normalized by $K$.

In the first lemma we collect some results which can be found in \cite{BSTT}, \cite{shu} and \cite{shu3}, respectively.

\begin{lem}\label{H^y} Let $m\geq 1$ and $G$ be a group. Then the following hold.
\begin{enumerate}
\item[$(i)$]  Let $y$ be an element of $G$ and $H$ a finitely generated subgroup. If $y^m$ is Engel, then $H^{\langle y\rangle}$ is finitely generated. {\rm \cite[Corollary 2]{BSTT}}
\item[$(ii)$] Assume that $G$ is generated by a finite subset $X$. If $x^m$ is Engel for all $x\in X$, then the derived subgroup of $G$ is finitely generated. {\rm \cite[Lemma~4]{BSTT}}
\item[$(iii)$] Let $X$ be a normal commutator-closed subset of $G$. Assume that $G$ is generated by finitely many elements of $X$. If $x^m$ is Engel for all $x\in X$, then each term of the derived series of $G$ is finitely generated. {\rm \cite[Corollary 5]{BSTT}}
\item[$(iv)$]  Assume that $G$ is generated by a normal commutator-closed set $X$ of elements of finite order. If $G$ contains a subgroup $N$ of finite index such that $X\cap N=1$, then $G$ is locally finite. {\rm \cite[Lemma\ 3.4]{shu}}
\item[$(v)$] Let $v$ be a multilinear commutator. Then there exists $k\geq 1$ such that every $\delta_k$-value in $G$ is a $v$-value. {\rm \cite[Lemma 4.1]{shu3}}
\end{enumerate}
\end{lem}

In the sequel the following results will play an important role.

\begin{lem}\label{HK}
Let $m\geq 1$ and $G$ be a group generated by two finitely generated subgroups $H$ and $K$. Assume that $K=\l y_1,y_2,\ldots,y_d\r$, where each $y_i^m$ is Engel in $G$. If $K$ is nilpotent, then $H^G$ is finitely generated.
\end{lem}

\begin{proof} Clearly, $H^G=H^K$. Let $c$ be the nilpotency class of $K$. We use induction on $d$ and $c$, respectively. If $d=1$, the claim follows from part $(i)$ of Lemma \ref{H^y}. Assume $d>1$. If $c=1$, then
$$H^K=(((H^{\l y_1\r})^{\l y_2\r})\ldots)^{\l y_d\r}$$
and so $H^K$ is finitely generated by Lemma \ref{H^y} $(i)$. Suppose that $c>1$. Let $N=\l y_1\r^K$ and $J=\l y_2,\ldots,y_d\r$. Thus $N$ is nilpotent of class at most $c-1$.  Moreover, as $N$ is a subgroup of the finitely generated nilpotent group $K$, it is generated by finitely many elements whose $m$-th powers are Engel. Now the induction hypothesis on $c$ implies that $H^N$ is finitely generated and, by induction on $d$, so is $(H^N)^J$. Since $K=JN$, it follows that $H^K$ is finitely generated, as desired.
\end{proof}

\begin{lem}\label{fn}
Let $G$ be a finite nilpotent group generated by a normal com\-mutator-closed set $X$ of elements of order dividing $e$. If $G$ satisfies a law $w\equiv 1$ and is $d$-generated for some $d\geq 1$, then its order is $\{d,e,w\}$-bounded.
\end{lem}

\begin{proof}
The proof is really very similar to that of Lemma 3.2 of \cite{shu1}. The modifications required in our situation are self-evident. We therefore omit the details.
\end{proof}

Let $G$ be a finite soluble group with Fitting subgroup $F(G)$. Recall that the Fitting height of $G$ is the least integer $h$ such that $F_h(G)=G$ where, as usual, $F_0(G)=1$ and $F_{i}(G)/F_{i-1}(G)=F(G/F_{i-1}(G))$ for any $i\geq1$. We denote by $\mathcal{N}^h$ the class of all finite soluble groups of Fitting height at most~$h$.

\begin{thm}\label{locfin}
Let $G$ be a group satisfying some group-law and suppose that $G$ is generated by a normal com\-mutator-closed set $X$ of elements of finite order. If $G$ is residually-$\mathcal{N}^h$, then $G$ is locally finite.
\end{thm}

\begin{proof}
Suppose first that $G$ is generated by finitely many elements of $X$ and argue by induction on $h$. If $h=1$, then $G$ is residually finite-nilpotent. Let $N$ be a normal subgroup of $G$ such that $G/N$ is finite nilpotent. By Lemma \ref{fn}, the order of $G/N$ is bounded by certain number which does not depend on $N$. It follows that $G$ is finite. Assume now that $h>1$ and let $H$ be the intersection of all normal subgroups $M$ of $G$ such that $G/M$ is nilpotent. The previous argument shows that $G/H$ is finite. It is easy to verify that $H$ is residually-$\mathcal{N}^{h-1}$. So, by the induction hypothesis, $X\cap H$ generates a normal locally finite subgroup $K$ of $G$. Now, looking at the quotient $G/K$, the claim is immediate from Lemma \ref{H^y} $(iv)$.

Let us now drop the assumption that $G$ is generated by finitely many elements from $X$ and choose any finitely generated subgroup $H$ of $G$. Then $H\leq K$ where $K$  is generated by finitely many elements from $X\cap K$ and, of course, the set $X\cap K$ is normal and commutator closed in $K$. Furthermore, $K$ is residually-$\mathcal{N}^h$. It follows from what we have shown in the previous paragraph that $K$ is finite. In particular, $H$ is finite and $G$ is locally finite.
\end{proof}

In what follows the intersection of all subgroups of finite index of a group $G$ is called the finite residual of $G$. We will require a corollary of the above theorem.

\begin{cor}\label{residual}
Let $G$ be a locally graded group generated by a normal com\-mutator-closed set $X$ of elements of finite order. Let $R$ be the finite residual of $G$ and suppose that $G/R$ is a residually-$\mathcal{N}^h$ group satisfying some group-law. Then $G$ is locally finite.
\end{cor}

\begin{proof}
Assume that $G$ is finitely generated. If $R=1$, the claim follows from Theorem \ref{locfin}. Let $R\neq1$. Then $G/R$ is finite by Theorem \ref{locfin} and so $R$ is finitely generated. Since $R$ is locally graded, we conclude that $R$ has a proper subgroup of finite index in $G$, a contradiction.
\end{proof}

In any group $G$ there exists a unique maximal normal locally nilpotent
subgroup (called the Hirsch-Plotkin radical) containing all normal locally nilpotent subgroups of $G$ \cite[12.1.3]{Rob}. In general, it is a subset of the set $L(G)$ of all (left) Engel elements \cite[12.3.2]{Rob}. However, it coincides with $L(G)$ whenever $G$ is soluble (Gruenberg, \cite[12.3.3]{Rob}), or $G$ satisfies the maximal condition (Baer, \cite[12.3.7]{Rob}), or $G$ has an ascending series with locally nilpotent factors (Plotkin, \cite[Exercise 12.3.7]{Rob}). More generally, by  Plotkin \cite{Pl}, this is true even if the group has an ascending series whose factors satisfy max locally (i.e., every finitely generated subgroup satisfies the maximal condition).

\begin{rem}\label{hall}
By Theorem $7.2.9$ of {\rm \cite{MHall}}, there are only finitely many subgroups of any given finite index in a free group of finite rank. It follows that in any finitely generated group there are only finitely many subgroups of that index.
\end{rem}

It was shown in \cite{LMS} that the quotient of a locally graded group over a normal locally nilpotent subgroup is again locally graded. This will be used in the proof of the following proposition.

\begin{pro}\label{v}
Let $k,m,n$ be positive integers, $v=\delta_k$ and $w=v^m$. If $G$ is a locally graded group in which all $w$-values are $n$-Engel, then the verbal subgroup $w(G)$ is locally nilpotent.
\end{pro}

\begin{proof} Denote by $X$ the set of all $\delta_k$-values in $G$ and choose a finitely generated subgroup $V$ of $w(G)$. Clearly, there exist finitely many $w$-values $y_1,\dots,y_d$ such that $V\leq\langle y_1,\dots,y_d\rangle$. Set $W=\langle y_1,\dots,y_d \rangle$. The proposition will be proved once it is shown that $W$ is nilpotent. We will use induction on $d$.

Let $H=\l y_1,\dots,y_{d-1},x\r$ where $x$ is a $\delta_k$-value such that $x^m=y_d$. Since $y_1,\dots,y_{d-1}$ and $x^m$ are Engel elements, we remark that every finite-by-soluble quotient of $H$ is an extension of a nilpotent group by a cyclic group of order dividing $m$ (\cite[12.3.3]{Rob}, \cite{Pl}). In particular, every finite quotient of $H$ is in the class $\mathcal{N}^2$. Set $N=\l x\r^H$. By the induction hypothesis the subgroup $\l y_1,\dots,y_{d-1}\r$ is nilpotent and so, by Lemma \ref{HK}, $N$ is finitely generated. More precisely, $N$ is generated by finitely many $\delta_k$-values. Furthermore, $X\cap N$ is a normal commutator-closed subset of $N$. Thus, by part $(iii)$ of Lemma \ref{H^y}, $N^{(i)}$ is finitely generated for every $i$. As a consequence, we obtain that $H^{(i)}$ is also finitely generated for every $i$. In fact $H/N^{(i)}$ is soluble and therefore it is polycyclic. Hence, $H^{(i)}/N^{(i)}$ is finitely generated and so is $H^{(i)}$. Let $R$ be the finite residual of $H$ and $S/R$ be the Hirsch-Plotkin radical of $H/R$. Then $w(H)R\leq S$, by \cite[Theorem~A]{BSTT}. Set $J=H/S$. By \cite{LMS}, $J$ is locally graded and, since $w(H)\leq S$, the $\delta_k$-commutators in $J$ have finite order dividing $m$. Corollary \ref{residual} now shows that $J^{(k)}$ is locally finite. Taking into account that $J^{(k)}$ is a homomorphic image of $H^{(k)}$ and that $H^{(k)}$ is finitely generated, we conclude that $J^{(k)}$ is finite. Thus $J$ is finite-by-soluble and $H/R$ is (locally nilpotent)-by-nilpotent-by-cyclic. But any Engel element in such a group lies in the Hirsch-Plotkin radical \cite{Pl}. So $H/R$ is nilpotent-by-(cyclic of order $m$) and $H^{(j)}\leq R$ for some $j$. On the other hand $H/H^{(j+1)}$ is residually finite because it has a nilpotent subgroup of finite index \cite[12.3.3]{Rob}, so that we must have $H^{(j)}=R=H^{(j+1)}$. Therefore $R$ is generated by finitely many elements. Since $G$ is locally graded, there exists a proper subgroup $T$ of $R$ of finite index $t$. By Remark \ref{hall}, $R$ contains only finitely many subgroups of index $t$. This implies that the intersection $M$ of such subgroups has finite index in $R$. Of course, $M$ is characteristic and $R/M$ is finite. We deduce that $H/M$ is nilpotent-by-cyclic \cite{Pl}. It follows that $H^{(j+1)}$ is a proper subgroup of $H^{(j)}$, which is impossible. This means that $R$ must be necessarily trivial. Thus $H$ is soluble and, since $W\leq H$, we conclude that $W$ is nilpotent \cite[12.3.3]{Rob}.
\end{proof}

\begin{lem}\label{rfbf}
Let $G$ be a finitely generated group such that $G'$ is finitely generated and residually finite. Then $G$ is residually finite.
\end{lem}

\begin{proof}
Let $1\neq g\in G$. Choose a subgroup $H$ of finite index in $G'$ such that $g\notin H$. Set $N=\bigcap_{x\in G} H^x$. By Remark \ref{hall}, $N$ has finite index in $G'$. Pass to the quotient $G/N$. Obviously this quotient is polycyclic-by-finite and hence residually finite. Therefore $G/N$ contains a subgroup $K/N$ of finite index such that $gN\notin K/N$. We see that the subgroup $K$ has finite index in $G$ and $g\notin K$. Thus, for any nontrivial element in $G$ there exists a subgroup of finite index that does not contain the element. This means that $G$ is residually finite.
\end{proof}

The next result is a generalization of Proposition 14 of \cite{BSTT}.

\begin{pro}\label{gamma}
Let $m,n$ be positive integers, $v$ a multilinear commutator word and $w=v^m$. Let $G$ be a group in which all $w$-values are $n$-Engel and assume additionally that $G$ is generated by finitely many Engel elements. If $\gamma_i(G)$ is residually finite for some $i\geq 1$, then $G$ is nilpotent.
\end{pro}

\begin{proof}
Let $G=\l x_1,\ldots,x_d\r$ where each $x_j$ is an Engel element. We use induction on $d$ and $i$, respectively. If $G$ is cyclic, the result is obvious so we assume that $d>1$. If $i=1$, the claim follows from Proposition 14 of \cite{BSTT}. Suppose $i=2$, that is, $G'$ is residually finite. Part $(ii)$ of Lemma \ref{H^y} tells us that $G'$ is finitely generated. In view of Lemma \ref{rfbf}, $G$ is residually finite and so we are back to the case $i=1$. Let $i>2$ and $j\in\{1,\ldots,d\}$. Set $H=\l x_j\r^G$. Then $K=\l x_1,\ldots,x_{j-1},x_{j+1},\ldots,x_d\r$ is nilpotent, by induction on $d$, so that $H=\l x_j \r^K$ is finitely generated by Lemma \ref{HK}. Now, $H$ is generated by finitely many Engel elements and its image in $G/\gamma_i(G)$ has nilpotency class at most $i-2$. Therefore $\gamma_{i-1}(H)$ is residually finite and, by induction on $i$, $H$ is nilpotent. We have shown that every generator $x_j$ of $G$ is contained in a normal nilpotent subgroup. Applying Fitting's Theorem (see, for instance, \cite[5.2.8]{Rob}), the proposition follows.
\end{proof}

We are now in a position to prove our main result.

\begin{proof}[\bf Proof of Theorem \ref{main}]
Recall that $v$ is a multilinear commutator word, $w=v^m$ and $G$ is a locally graded group in which all $w$-values are $n$-Engel. We will prove that $w(G)$ is locally nilpotent.

By Lemma \ref{H^y} $(v)$, there exists $k\geq 1$ such that every $\delta_k$-value is a $v$-value. Therefore, by Proposition \ref{v}, the verbal subgroup corresponding to $w_k=\delta_k^m$ is locally nilpotent. Next, we will prove that $w(G/w_k(G))$ is locally nilpotent. We remark that $G/w_k(G)$ is itself locally graded \cite{LMS}. So, without loss of generality, we may assume that all $\delta_k$-commutators in $G$ have finite order dividing $m$. Let $w_1,\dots,w_d$ be $w$-values of $G$ and set $W=\langle w_1,\dots,w_d\rangle$. Since every soluble image of $W$ is nilpotent \cite[12.3.3]{Rob}, there exists $i\geq 1$ such that $\gamma_i(W)\leq W^{(k)}$. Let $R$ be the finite residual of $W^{(k)}$. Then $R\lhd W$ and $\gamma_i(W/R)$ is residually finite. By Proposition \ref{gamma}, it follows that $W/R$ is nilpotent. In particular $W^{(k)}/R$ is nilpotent and generated by elements of finite order, so that $W^{(k)}$ is locally finite by Corollary \ref{residual}. As $W/W^{(k)}$ is nilpotent, according to Plotkin \cite{Pl}, we conclude that $W$ is nilpotent. This proves that $w(G/w_k(G))$ is locally nilpotent. Thus, $w(G)$ has a normal series $1\leq w_k(G)\leq w(G)$ all of whose factors are locally nilpotent. Hence, by Plotkin \cite{Pl}, $w(G)$ is locally nilpotent, as required.
\end{proof}

\vspace {0.5cm}

\noindent{\bf Aknowledgements.} The authors wish to thank the anonymous referee for useful suggestions.


\begin{thebibliography}{10}
\bibitem{BSTT} R. Bastos, P. Shumyatsky, A. Tortora and M. Tota, {\it On groups admitting a word whose values are Engel},
Internat. J. Algebra Comput. {\bf 23} (2013), 81--89.

\bibitem{MHall} M. Hall, Jr., {\it The theory of groups}, The Macmillan Co., New York, 1959.

\bibitem{kr} Y. Kim and A.\,H. Rhemtulla,
{\it On locally graded groups}, Groups-Korea '94 (Pusan), 189--197, de Gruyter, Berlin, 1995.

\bibitem{LMS} P. Longobardi, M. Maj and H. Smith,
{\it A note on locally graded groups}, Rend. Sem. Mat. Univ. Padova {\bf 94} (1995) 275--277.

\bibitem{Pl} B.\,I. Plotkin, {\it Radicals and nil-elements in groups},
Izv. Vys$\check{\rm s}$ U$\check{\rm c}$ebn. Zaved. Matematika {\bf 1} (1958), 130--135.

\bibitem{Rob} D.\,J.\,S. Robinson,
\textit{A course in the theory of groups}, 2nd edition, Springer-Verlag, New York, 1996.

\bibitem{shu0} P. Shumyatsky, {\it On residually finite groups in which commutators are Engel}, Comm. Algebra {\bf 27} (1999), 1937--1940.

\bibitem{shu2} P. Shumyatsky, {\it Applications of Lie ring methods to group theory}, Nonassociative algebra and its
applications (S$\tilde{\rm a}$o Paulo, 1998), 373--395, Lecture Notes in Pure and Appl. Math. {\bf 211}, Dekker, New York, 2000.

\bibitem{shu3} P. Shumyatsky, {\it Verbal subgroups in residually finite groups}, Q. J. Math. {\bf 51} (2000), 523--528.

\bibitem{shu1} P. Shumyatsky, {\it On varieties arising from the solution of the restricted Burnside problem},
J. Pure Appl. Algebra {\bf 171} (2002), no. 1, 67--74.

\bibitem{shu} P. Shumyatsky, {\it Elements of prime power order in residually finite groups},
Int. J. Algebra Comput. {\bf 15} (2005), no. 3, 571--576.

\bibitem{w} J.\,S. Wilson, {\it Two-generator conditions for residually finite
groups}, Bull. London Math. Soc. {\bf 23} (1991), 239--248.

\bibitem{ze1} E.\,I. Zelmanov,  {\it Solution of the restricted
Burnside problem for groups of odd exponent}, Math. USSR-Izv.
{\bf 36} (1991), no. 1, 41--60.

\bibitem{ze2} E.\,I. Zelmanov, {\it Solution of the restricted
Burnside problem for $2$-groups}, Math. USSR-Sb. {\bf 72} (1992), no. 2, 543--565.

\end{thebibliography}
\end{document}